\documentclass[10pt,a4paper]{article}

\usepackage{amsfonts}
\usepackage{amsmath}
\usepackage{amssymb}
\usepackage{amsthm}
\usepackage{bussproofs}

\input xy
\xyoption{all}

\newtheorem{definition}{Definition}

\newtheorem{lemma}[definition]{Lemma}
\newtheorem{theorem}[definition]{Theorem}
\newtheorem{corollary}[definition]{Corollary}

\theoremstyle{remark}

\newtheorem{example}[definition]{Example}

\def\NN{\mathbb{N}}
\def\QQ{\mathbb{Q}}

\def\RR{\mathbb{R}}

\def\Types{{\bf T}}

\def\BB{\mathbb{B}}
\newcommand{\True}{0}
\newcommand{\False}{1}

\newcommand{\PLomega}{{\sf PL}^\omega}
\newcommand{\PAomega}{{\sf PA}^\omega}
\newcommand{\PRA}{{\sf PRA}}

\newcommand{\FC}{{\sf FC}}

\newcommand{\DC}{{\sf DC}}

\newcommand{\AC}{{\sf AC}_0}
\newcommand{\ACPi}[1]{\Pi_{#1}\mbox{-}\AC}

\newcommand{\CASig}{\Sigma_1\mbox{-}{\sf CA}}

\newcommand{\DP}{{\sf DP}}

\newcommand{\WKL}{{\sf WKL}}

\newcommand{\BW}{{\sf BW}}

\newcommand{\systemT}{{\rm T}}
\newcommand{\SBR}{{\sf SBR}}

\newcommand{\nexts}{{\rm next}}

\newcommand{\cZero}{{\bf 0}}

\newcommand{\emptyseq}{\langle \, \rangle}

\newcommand{\eqleft}[1]{\begin{itemize} \item[] $#1$ \end{itemize}}
\newcommand{\pair}[1]{\langle #1 \rangle}

\newcommand{\initSeg}[2]{[#1](#2)}
\newcommand{\ext}[1]{\widehat{#1}}

\newcommand{\selEmb}[1]{\overline{#1}}

\newcommand{\nd}[3]{|{#1}|^{#2}_{#3}}
\newcommand{\nt}[1]{{#1}^N}

\newcommand{\pred}{{\; \varphi}}
\newcommand{\negpred}{{\; \neg\varphi}}

\newcommand{\mt}{T}

\newcommand{\Rec}[3]{{\sf R}_{#1}(#2)(#3)}

\newcommand{\Prod}[3]{{\sf EPS}_{#1}^{#2}({#3})}
\newcommand{\Prods}{{\sf EPS}}

\begin{document}

\title{A Game-Theoretic Computational Interpretation of Proofs in Classical Analysis}

\author{Paulo Oliva and Thomas Powell}

\date{}

\maketitle

\begin{abstract} It is shown in \cite{EO(2010A),EO(2011A)} that a functional interpretation of proofs in mathematical analysis can be given by the product of selection functions, a mode of recursion that has an intuitive reading in terms of the computation of optimal strategies in sequential games. We argue that this result has genuine practical value by interpreting some well-known theorems of mathematics and demonstrating that the product gives these theorems a natural computational interpretation that can be clearly understood in game theoretic terms.
\end{abstract}

\section{Introduction}
\label{sec-intro}

Over the last century, mathematicians and computer scientists have become increasingly interested in understanding the computational content of mathematical proofs.

A central feature of modern mathematics is the use of non-constructive methods that allow us to reason about infinitary objects without providing any computational justification. In the 1920's Hilbert's program broadly addressed the task of understanding non-constructive mathematics in computational terms, which led to the development of important proof-theoretic techniques such as cut-elimination, the $\varepsilon$-method, and proof interpretations. These were used to obtain significant foundational results such as relative consistency proofs for arithmetic and analysis.

In recent decades these metamathematical devices whose roots lie in foundational problems have been employed more directly towards the extraction of programs from non-constructive proofs. This shift of emphasis has its origins in the fundamental work of Kreisel on the `unwinding' of proofs \cite{Kreisel(51), Kreisel(52)}, and has now become the focus of a considerable amount of research in logic and computer science.

Among the most effective tools for extracting constructive information from proofs are the \textit{proof interpretations}, which include Freidman's A-translation \cite{Friedman(78)} and G\"{o}del's dialectica interpretation \cite{Goedel(58)}. The latter in particular is central to the highly successful \textit{proof mining} program (see Kohlenbach \cite{Kohlenbach(2008)}), in which the analysis of proofs using its monotone variant has led to the development of quite general meta-theorems guaranteeing the extraction of effective uniform bounds from theorems in analysis. 

While proof interpretations have been widely applied in logic and computer science, the \textit{qualitative (computational) behaviour} of their output has received relatively little attention. Indeed, the operational semantics of programs extracted from even relatively simple classical proofs are often very difficult to understand. This is mainly due to two factors:
\begin{enumerate}
	\item {\bf Higher-order computation}. Even when computing a witness of type $\NN$ the constructions involved will work on much higher types, usually types $1$ and $2$. G\"odel's primitive recursor itself even for the lowest type is already an object of type $2$.
	\item {\bf Syntactic nature of proof interpretations and translations}. Extracted programs tend to be hidden beneath a complex layer of syntax that generally accompanies formal translations on proofs. 
\end{enumerate}
Moreover, relatively little work has gone into addressing these issues because more often than not proof interpretations are a means to an end -- be it a consistency proof or the extraction of a uniform bound -- and a qualitative understanding of their output is simply irrelevant.

Nevertheless, the idea of stripping functional interpretations of their syntax and appreciating how they work from a \textit{mathematical} perspective is an interesting one. It has been observed by Gaspar and Kohlenbach \cite{Gaspar(2010),Kohlenbach(2008)} that the kind of logical manipulations carried out by the dialectica interpretation is closely related to the so-called \emph{correspondence principle} between `soft' and `hard' analysis discussed by T. Tao in \cite{Tao(2007),Tao(2008)}. In this sense one could potentially view functional interpretations as devices that transform classical proofs into constructive proofs of a `finitized' form of the original theorem, although actually translating their output into what a mathematician would consider a proof seems far from straightforward.

In recent work \cite{EO(2010A),EO(2011A),OP(2011B)} the authors and M. Escard{\'o} have sought to better understand programs extracted by the dialectica interpretation. Here it is shown that the dialectica interpretation of the key combination of classical logic and countable choice can be realized by the \emph{product of selection functions} (as opposed to the usual bar recursion of Spector \cite{Spector(62)}), an intuitive mode of computation that can be understood as computing optimal strategies in a class of sequential games.

Consider, for instance, an $\exists \forall$-theorem $\exists x^X \forall y^Y A(x, y)$, with $A(x, y)$ a decidable predicate, and $x$ and $y$ having types $X$ and $Y$ respectively. We shall think of $X$ as a set of available \emph{moves}, and $Y$ as the set of possible \emph{outcomes} of a game. The predicate $A(x, y)$ is then understood as prescribing what are the good outcomes $y$ given any particular move $x$. The theorem then says that there exists a single move for which all possible outcomes are considered good. Now, if the theorem has been proven classically, such a move will be shown to exist but it might not be effectively computable. What we should do then is to consider the `constructive' equivalent of the theorem via the negative translation, namely $\neg \neg \exists x^X \forall y^Y A(x, y)$. This can also be put in the form $\exists \forall$, and in fact that is precisely what the dialectica interpretation does. In this case we would obtain the (classically) equivalent theorem
\[ \exists \varepsilon^{(X \to Y) \to X} \forall p^{X \to Y} A(\varepsilon p, p(\varepsilon p)). \]
Although $x \colon X$ might not be effectively computable, it turns out that the \emph{selection function} $\varepsilon \colon (X \to Y) \to X$ is. Moreover, we can extend our game-theoretic reading and view $p \colon X \to Y$ as a mapping from moves to outcomes. What the selection function $\varepsilon$ does is to pick, for any given such mapping $p$,  a move $x = \varepsilon p$ whose corresponding outcome according to $p$, namely $y = p x$ is a good outcome for $x$. 

Now suppose we are given a countable family of $\forall\exists$-predicates $\exists x\forall yA_n(x,y)$ interpreted by a sequence of selection functions $(\varepsilon_n)$. By classical countable choice there exists a sequence $f\colon\NN\to X$ satisfying $\forall n, y A_n(fn, y)$. The dialectica interpretation of $\neg\neg \exists f\forall n,y A_n(fn,y)$ states that for any given functions $q \colon X^\NN \to Y$ and $\omega \colon X^\NN \to \NN$ there exists a functional $f$ such that  
\[ \forall n \leq \omega f \, A_n(f  n, q f). \]
Therefore, thinking of each $A_n(x, y)$ as prescribing the ``good" pairs of move-outcome for round $n$, the task above corresponds in  finding a sequence of moves $f$ which leads to an outcome $y = q f$ that is considered good at all rounds up to point $\omega f$. We will see that the product of the selection functions $(\varepsilon_n)$ calculates such $f$, and this construction can be viewed as the calculation of an optimal strategy in a sequential game whose ``goal" at round $n$ -- given by the selection function $\varepsilon_n$ -- is to pick a move with a good outcome for the predicate $A_n$. 

The aim of this article is to demonstrate that, \emph{in practise}, program extraction using the product of selection of functions and its game-theoretic semantics leads to a much better appreciation of the constructive content of proofs in analysis. We illustrate this using a number of well known classical theorems, many of which have been extensively analysed by proof theorists. In particular we include a detailed analysis of a proof of the Bolzano-Weierstrass theorem, the interpretation of which is by no means trivial, but from which we are nevertheless able to extract a program that can be given a clear description in the language of sequential games.

In the course of the paper our aim is to portray the dialectica interpretation as an intelligent translation whose output can be read and understood in mathematical terms. As such we endeavour to phrase the higher type functionals that arise from the interpretation using a more informal vocabulary. This approach owes a lot to the aforementioned work by Gaspar and Kohlenbach, and it is hoped that our work will complement theirs in forming another small step towards understanding the mathematical significance of proof interpretations.  

\subsection{Preliminaries}
\label{subsec-prelim}

We work in the language of Peano arithmetic in all finite types $\PAomega$. The finite types $\Types$ contain a basic type $\NN$ and whenever $X, Y \in \Types$ then $X \to Y \in \Types$, i.e.
\eqleft{\Types = \{\NN,\NN\rightarrow\NN,(\NN\rightarrow\NN)\rightarrow\NN,\ldots\}.}
Closely related to $\PAomega$ is G\"{o}del's system $\systemT$ of primitive recursive functionals of finite type. This quantifier-free calculus is essentially primitive recursive arithmetic $\PRA$ with the schema of recursion extended to all types $X \in \Types$, i.e.
\begin{equation} \label{t-rec}
\Rec{n}{h}{g} \stackrel{X}{=}
\left\{
\begin{array}{ll}
	g & {\rm if } \; n = 0 \\[2mm]
	h_n(\Rec{n-1}{h}{g}) & {\rm if} \; n > 0.
\end{array}
\right.
\end{equation}
For full details of these theories the reader is referred to \cite{Avigad(98)}. We make informal use of types like $\BB$, $\QQ$, finite sequences types $X^\ast$ etc. as elements of these types can be encoded as elements of a suitable type in $\Types$. \\[2mm]
{\bf Notation}. We make use of the following abbreviations:
\begin{description}
\item[]$0_X$ is the obvious canonical zero of type $X$.
\item[]$s\ast t$ is the concatenation of sequences $s$ and $t$.
\item[]$s \preceq t$ for ``$s$ is a prefix of $t$".
\item[]$\ext{s} \equiv s * \cZero^{X^\NN}$ the canonical infinite extension of a finite sequence $s$.
\item[]$\initSeg{\alpha}{n}\equiv\pair{\alpha 0,\ldots,\alpha (n-1)}$ is the initial segment of $\alpha$ of length $n$.
\item[]$\mu n\leq N \; . \; P(n)$ is the bounded search operator that returns the least $n\leq N$ satisfying the decidable predicate $P(n)$ if one exists, or $N$ otherwise.
\end{description}

\subsection{The dialectica interpretation}

We assume that the reader is familiar with G\"{o}del's dialectica interpretation (details of which are covered in full in \cite{Avigad(98),Kohlenbach(2008)}), although we recall below a few basic facts to familiarise the reader with our notation and terminology.

The dialectica interpretation maps formulas $A$ of some specified theory $\mathcal{S}$ to a decidable binary relation $\nd{A}{x}{y}$ definable in a specified quantifier-free system of functionals ${\sf F}$. The canonical instance of this mapping is when $\mathcal{S}$ is Heyting arithmetic in all finite types and ${\sf F}$ is G\"{o}del's system $\systemT$.

In $\nd{A}{x}{y}$ we have that $x$ and $y$ stand for (possibly empty) tuples of objects of finite type. We think of  $x$ as the \textit{witnessing} variables and $y$ as the \textit{challenge} variables. The intuition is that $A$ is logically equivalent to $\exists x\forall y\nd{A}{x}{y}$. The translation is formally defined as follows:

\begin{definition}[G\"odel's dialectica interpretation] \label{dialectica} For atomic formulas $P$ we set $\nd{P}{}{}:=P$, with $x$ and $y$ both empty tuples. Assuming that we have already defined $\nd{A}{x}{y}$ and $\nd{B}{u}{v}$, we define
\eqleft{
\begin{array}{lcl}
\nd{A\wedge B}{x,v}{y,w} & := & \nd{A}{x}{y}\wedge\nd{B}{v}{w} \\[2mm]
\nd{A\vee B}{x,v,b}{y,w} & := & (b=\True \wedge\nd{A}{x}{y})\vee(b = \False \wedge\nd{A}{v}{w}) \\[2mm]
\nd{A\to B}{f,g}{x,w} & := & \nd{A}{x}{gxw}\to\nd{B}{fx}{w} \\[2mm]
\nd{\forall zA(z)}{f}{y,z} & := & \nd{A(z)}{fz}{y} \\[2mm]
\nd{\exists zA(z)}{x,z}{y} & := & \nd{A(z)}{x}{y}.
\end{array}
}
We say that $\mathcal{S}$ is (dialectica) interpreted in $\sf F$ if whenever $\mathcal{S}\vdash A$ we can construct some $t\in\sf{F}$ such that $\sf{F} \vdash \nd{A}{t}{y}$.
\end{definition}

In order to interpret \textit{classical} theories, the dialectica interpretation is typically composed with a negative translation\footnote{As in \cite{Kohlenbach(2008)} we adopt Kuroda's variant of the negative translation.} $N$ to form the so-called ND interpretation. In the remainder of the paper, by \emph{functional interpretation} we specifically mean the ND interpretation. A classical theory $\mathcal{T}$ has a functional interpretation in $\sf F$ if whenever $\mathcal{T}\vdash A$ we can construct some $t\in\sf{F}$ satisfying $\forall y\nd{\nt{A}}{t}{y}$.

In his original paper on the dialectica interpretation, G\"{o}del proved that Peano arithmetic has a functional interpretation in the primitive recursive functionals of finite type $\systemT$. Later, Spector extended G\"{o}del's result to classical analysis by realizing the dialectica interpretation of the negative translation of the axiom of countable choice $\AC$ with a novel, but rather abstruse form of recursion called \emph{bar recursion} $\SBR$.

\begin{theorem}\label{thm-soundness} The following are well-known:

\begin{itemize}

\item[(a)]$\PAomega$ has a functional interpretation in $\systemT$ \emph{(G\"{o}del \cite{Goedel(58)})}.

\item[(b)]$\PAomega+\AC$ has a functional interpretation in $\systemT+\SBR$ \emph{(Spector \cite{Spector(62)})}.

\end{itemize}\end{theorem}

The main purpose of this paper is to show that these soundness theorems can be reformulated in terms of the product of selection functions, and that this reformulation is better suited towards understanding the behaviour of programs extracted by the dialectica interpretation.

\subsection{Outline of article}
\label{subsec-outline}

We begin in Section \ref{sec-selection} by introducing the product of selection functions and showing that it can be characterised as an operation that computes optimal strategies in sequential games.

In the main part of the paper we then discuss how the language of selection functions is well suited to capturing the way in which the dialectica interpretation works, and in particular the product of selection functions directly interprets countable choice.

We then present a short case study (Section \ref{sec-bolzano}) in which we extract a program from a proof of the Bolzano-Weierstrass theorem via the product of selection functions and demonstrate that our program has a clear game-theoretic semantics.

We conclude by briefly discussing some of the problems we face in gaining a more intuitive understanding of functional interpretations, and outline some potential directions for further research.

\section{Selection Functions and Sequential Games}
\label{sec-selection}

This section constitutes a brief overview of work that is presented in full elsewhere, e.g. the reader is referred to the original paper \cite{EO(2009)} or a recent survey \cite{EO(2011A)} for a more detailed treatment. 

A \emph{selection function} is defined to be any element of type $(X\to R)\to X$ (as in \cite{EO(2009)} we abbreviate this type to $J_RX$). Closely related to a selection function $\varepsilon\colon J_RX$ is its corresponding \emph{quantifier} $\bar\varepsilon\colon (X\to R)\to R$ defined by $\bar\varepsilon p:=p(\varepsilon p)$. The intuition is to view $\varepsilon$ as a selector that given a function $p\colon X\to R$ picks a particular element of $\varepsilon p$ of $X$ that \emph{attains} its quantifier $\bar\varepsilon$, as the following examples illustrate. 
\begin{example}\label{ex-selection}\begin{itemize}

\item[(a)] The canonical example of a selection function and its associated quantifier is when $R$ forms a set of truth values e.g. $R = \BB$. Hilbert's epsilon term of type $X$, $\varepsilon_X\colon J_\BB X$ is a selection function whose corresponding quantifier is just the usual existential quantifier $\exists_X$ for predicates over type $X$, since by definition we have $$\exists x^X \; p(x)\Leftrightarrow p(\varepsilon_Xp).$$

\item[(b)] By the mean value theorem there exists a selection function $\varepsilon\colon J_{[0,1]}\RR$ such that for any continuous function $p\colon[0,1]\to \RR$ we have $$p(\varepsilon p)=\int_0^1 p(x) dx.$$ Its corresponding quantifier is the operator $\int^1_0$. 

\item[(c)] Assume we are given a position in a game where we have to pick a move in $X$. A \emph{strategy} for that position can be defined by a selection function $\varepsilon\colon J_RX$ determining an optimal move for each given mapping $p \colon X \to R$ of possible moves $x \in X$ to corresponding outcomes $p(x) \in R$. 

\end{itemize}\end{example}
The theory of selection functions and quantifiers forms the basis of \cite{EO(2010A),EO(2009),EO(2011A)}. One of the main achievements of these papers has been to define a product operation on selection functions (along with a corresponding operation on quantifiers which we do not discuss further here). They demonstrate that the product of selection functions is an extremely versatile construction that appears naturally in several different areas of mathematics and computer science, such as fixed point theory (Beki\v{c}'s lemma), algorithms (backtracking), game theory (backward induction) and, as we also discuss in Section \ref{sec-main}, proof theory. 

In the remainder of the section we define (following \cite{EO(2009)}) the product of selection functions, and explain how this procedure can be best understood via the computation of optimal strategies in a certain class of sequential games. 

\begin{definition}[Binary product of selection functions \cite{EO(2009)}] Given a selection function $\varepsilon \colon J_RX$ and family of selection functions $\delta_x  \colon J_RY$ and a predicate $q\colon X\times Y\to R$, let
\begin{align*}A[x^X] &\stackrel{Y}{:=} \delta(x, \lambda y.q(x,y)) \\ a &\stackrel{X}{:=} \varepsilon(\lambda x.q(x,A[x])).\end{align*}
The binary product $\varepsilon\otimes\delta$ is another selection function, of type $J_R(X\times Y)$, defined by
\[(\varepsilon\otimes\delta)(q) \stackrel{X\times Y}{:=} \pair{a,A[a]}.\]
If $\delta$ is independent of $x$ we can this the \emph{simple} product of selection functions. The general case is then also called the \emph{dependent} product of selection functions.
\end{definition}

The binary product constructs a composite selection function on the type $X\times Y$ in the obvious way:

\begin{example}\label{ex-product} Continuing from Example \ref{ex-selection} we have: \begin{itemize}

\item[(a)] It is easy to show that the product of $\varepsilon$ operators $\varepsilon_X\otimes\varepsilon_Y$ is an epsilon operator of type $X\times Y$ in the sense that $$\exists x^X\exists y^Y q(x,y)\Leftrightarrow q((\varepsilon_X\otimes\varepsilon_Y)(q)).$$

\item[(b)] Given a continuous function $q\colon [0,1]^2\to\RR$ we have $$q((\varepsilon\otimes\varepsilon)(q))=\int^1_0\int^1_0 q(x,y) \; dxdy.$$

\item[(c)] Given strategies $\varepsilon_0$, $\varepsilon_1$ for each round in a two round sequential game with outcome function $q\colon X_0\times X_1\to R$, then $(\varepsilon_0\otimes\varepsilon_1)(q)$ forms an \emph{strategy} for the game which is ``compatible" with the local strategies $\varepsilon_0$ and $\varepsilon_1$. This key instance of the product is discussed in more detail below.

\end{itemize}\end{example}

As described in \cite{EO(2009)}, we can iterate the binary product of selection functions a finite or an unbounded number of times, where the length of the iteration is dependent on the output of the product in the following sense.

\begin{definition}[Iterated product of selection functions \cite{EO(2009)}] \label{unbounded} Suppose we are given a family of selection functions $(\varepsilon_s\colon J_R X)$. The \textit{explicitly controlled unbounded product} of the selection functions $\varepsilon_s$ is defined by the recursion schema
\begin{equation} \label{Prod}
\Prod{s}{\omega}{\varepsilon}(q) \stackrel{X^{\NN}}{=}
\left\{
\begin{array}{ll}
	{\bf 0} & {\rm if } \; \omega(\ext{s})<|s| \\[2mm]
	(\varepsilon_s \otimes \lambda x . \Prod{s * x}{\omega}{\varepsilon})(q) & {\rm otherwise}
\end{array}
\right.
\end{equation}
where $s \colon X^*$, $q \colon X^{\NN} \to R$ and $\omega \colon X^{\NN} \to \NN$.
\end{definition}

The functional $\omega$ acts as a control, terminating the procedure once it has produced a sequence $s$ satisfying $\omega(\ext{s})<|s|$. The unbounded product is total in any model of bar recursion, which in particular must admit \emph{Spector's condition}:
\begin{equation*}
\forall\omega^{X^\NN\to\NN},\alpha^{X^\NN}\exists n \left( \omega(\widehat{\initSeg{\alpha}{n}})<n \right).
\end{equation*}
These include the models of continuous functionals and the majorizable functionals. On the other hand, when $\omega$ is a constant function, say $\omega \alpha = n$, this corresponds to a finite iteration of the binary product and this restricted instance of the product is definable in the primitive recursive functionals and therefore exists in any model of system $\systemT$.

By unwinding the definition of the binary product in (\ref{Prod}) we obtain an equivalent equation
\begin{equation} \label{Prod-var}
\Prod{s}{\omega}{\varepsilon}(q) \stackrel{X^{\NN}}{=}
\left\{
\begin{array}{ll}
	{\bf 0} & {\rm if } \; \omega(\ext{s})<|s| \\[2mm]
	a_s * \Prod{s * a_s}{\omega}{\varepsilon}(q_{a_s}) & {\rm otherwise}
\end{array}
\right.
\end{equation}
where $a_s = \varepsilon_s(\lambda x . \selEmb{\Prod{s * x}{\omega}{\varepsilon}}(q_x))$ and $q_a(\alpha) = q(a * \alpha)$. 

For fixed $\omega, \varepsilon$ and $q$ one can think of $\Prod{s}{\omega}{\varepsilon}(q)$ as computing an infinite extension to any given finite sequence $s$. The key property of $\Prods$ is that the infinite extension of an initial segment $\initSeg{\alpha}{n}$ of a previous infinite extension $\alpha$ is identical to the original infinite extension. Formally:

\begin{lemma}[cf. \cite{Spector(62)}, lemma 1] \label{main-lemma} Let $\alpha = \Prod{s}{\omega}{\varepsilon}(q)$. For all $n$, 
\begin{equation} \label{main-lemma-eq}
\alpha = \initSeg{\alpha}{n} * \Prod{s * \initSeg{\alpha}{n}}{\omega}{\varepsilon}(q_{\initSeg{\alpha}{n}}).
\end{equation}
\end{lemma}
\begin{proof} By induction on $n$. If $n = 0$ this follows by the definition of $\alpha$. Assume (\ref{main-lemma-eq}) holds for $n$, we wish to show it also holds for $n+1$. Consider two cases. \\[2mm]
(a) If $\omega(s * \initSeg{\alpha}{n} * \cZero) < n$ then $\Prod{s * \initSeg{\alpha}{n}}{\omega}{\varepsilon}(q_{\initSeg{\alpha}{n}}) = \cZero^{X^\NN}$. By induction hypothesis $\alpha = \initSeg{\alpha}{n} * \cZero$, so that $\alpha(n) = \cZero^X$. Therefore $\alpha = \initSeg{\alpha}{n+1} * \cZero$, which, by extensionality, implies $\omega(s * \initSeg{\alpha}{n+1} * \cZero) = \omega(s * \initSeg{\alpha}{n} * \cZero) < n < n+1$. Hence, $\Prod{s * \initSeg{\alpha}{n+1}}{\omega}{\varepsilon}(q_{\initSeg{\alpha}{n+1}}) = \cZero^{X^\NN}$ so that
\eqleft{\initSeg{\alpha}{n+1} * \Prod{s * \initSeg{\alpha}{n+1}}{\omega}{\varepsilon}(q_{\initSeg{\alpha}{n+1}}) = \initSeg{\alpha}{n+1} * \cZero = \initSeg{\alpha}{n} * \cZero = \alpha.}
(b) If $\omega(s * \initSeg{\alpha}{n} * \cZero) \geq n$, then
\[ \alpha \stackrel{\textup{(IH)}}{=} \initSeg{\alpha}{n} * \Prod{s * \initSeg{\alpha}{n}}{\omega}{\varepsilon}(q_{\initSeg{\alpha}{n}}) \stackrel{(\ref{Prod-var})}{=} \initSeg{\alpha}{n} * c * \Prod{s * \initSeg{\alpha}{n} * c}{\omega}{\varepsilon}(q_{\initSeg{\alpha}{n} * c}), \]
where $c = \varepsilon_{s * \initSeg{\alpha}{n}}(\lambda x . \selEmb{\Prod{s * \initSeg{\alpha}{n} * x}{\omega}{\varepsilon}}(q_{s * \initSeg{\alpha}{n} * x}))$. Hence, $\alpha(n) = c$. Therefore
\[ \alpha = \initSeg{\alpha}{n+1} * \Prod{s * \initSeg{\alpha}{n+1}}{\omega}{\varepsilon}(q_{\initSeg{\alpha}{n+1}}). \] 
\end{proof}

This lemma is the main building block behind the proof of the following fundamental theorem about $\Prods$.

\begin{theorem}[Main theorem on $\Prods$] \label{eps-main} Let $q \colon X^\NN \to R$ and $\omega \colon X^\NN \to \NN$ and $\varepsilon_s \colon J_R X$ be given. Define
\eqleft{
\begin{array}{lcl}
\alpha & \stackrel{X^\NN}{=} & \Prod{\emptyseq}{\omega}{\varepsilon}(q) \\[2mm]
p_s(x) & \stackrel{R}{=} & \selEmb{\Prod{s * x}{\omega}{\varepsilon}}(q_{s * x}).
\end{array}
}
For $n \leq \omega(\alpha)$ we have
\begin{equation}\label{eqn-equilibrium}
\begin{array}{lcl}
	\alpha(n) & \stackrel{X}{=} & \varepsilon_{\initSeg{\alpha}{n}}(p_{\initSeg{\alpha}{n}}) \\[2mm]
	q \alpha & \stackrel{R}{=} & \selEmb{\varepsilon_{\initSeg{\alpha}{n}}}(p_{\initSeg{\alpha}{n}}).
\end{array}
\end{equation}
\end{theorem}
\begin{proof} Assume $n \leq \omega(\alpha)$. We argue that $(*) \; n \leq \omega(\initSeg{\alpha}{n} * \cZero)$. Otherwise, assuming $n > \omega(\initSeg{\alpha}{n} * \cZero)$ we would have, by Lemma \ref{main-lemma}, that $\alpha = \initSeg{\alpha}{n} * \cZero$. And hence, $n > \omega(\initSeg{\alpha}{n} * \cZero) = \omega(\alpha) \geq n$, which is a contradiction. Hence, we have that
\eqleft{
\begin{array}{lcl}
	\alpha(n) 
		& \stackrel{\textup{L}\ref{main-lemma}}{=} & \Prod{\initSeg{\alpha}{n}}{\omega}{\varepsilon}(q_{\initSeg{\alpha}{n}})(0) \\[1mm]
		& \stackrel{(\ref{Prod-var}) + (*)}{=} & \varepsilon_{\initSeg{\alpha}{n}}(\lambda x . \selEmb{\Prod{\initSeg{\alpha}{n} * x}{\omega}{\varepsilon}}(q_{\initSeg{\alpha}{n} * x})) \\[2mm]
		& = & \varepsilon_{\initSeg{\alpha}{n}}(p_{\initSeg{\alpha}{n}}),
\end{array}
}
by the definition of $p_s$. For the second equality, we have
\eqleft{
\begin{array}{lcl}
	q \alpha & \stackrel{\textup{L}\ref{main-lemma}}{=} & q_{\initSeg{\alpha}{n+1}}(\Prod{\initSeg{\alpha}{n+1}}{\omega}{\varepsilon}(q_{\initSeg{\alpha}{n+1}})) \\[2mm]
	& = & p_{\initSeg{\alpha}{n}}(\alpha(n)) \\[2mm]
	& = & \selEmb{\varepsilon_{\initSeg{\alpha}{n}}}(p_{\initSeg{\alpha}{n}}),
\end{array}
}
where the last equality uses that $\alpha(n) = \varepsilon_{\initSeg{\alpha}{n}}(p_{\initSeg{\alpha}{n}})$ is already shown.
\end{proof}

Theorem \ref{eps-main} characterises the product of selection functions as computing a sequence $\alpha$ that forms a kind of sequential equilibrium between the selection functions -- expressed by the equations (\ref{eqn-equilibrium}) -- up to a point $\omega\alpha$ parametrised by $\alpha$ itself. The significance of the product is that such equilibria appear naturally in a variety of contexts. In the following we outline perhaps the most illuminating of these contexts, namely the theory of sequential games.

\subsection{Sequential games and optimal strategies}
\label{games}

One of the most remarkable property of $\Prods$ is that it computes optimal strategies in a certain class of sequential games. The reader is encouraged to consult \cite{EO(2011A)} in conjunction with the relatively concise discussion here. 

As in this article we only consider games (in the sense of \cite{EO(2011A)}) where the quantifiers are attainable, we shall incorporate this restriction in the definition of the game itself.

\begin{definition}[Sequential games of unbounded length, \cite{EO(2011A)}] \label{seqgames} The \emph{type} of a game is given by a pair $(X, R)$ where
\begin{itemize}
	\item $X$ is the set of possible \emph{moves} at each round.
	\item $R$ is the set of possible outcomes of the game.
\end{itemize}
A finite sequence $s \colon X^*$ shall be thought of as a \emph{position} in the game determined by the first $|s|$ moves. An infinite sequence $\alpha \colon X^\NN$ is called a \emph{play} of the game. An unbounded sequential game of type $(X, R)$ is a triple $(\varepsilon, q, \omega)$ where

\begin{itemize}
\item $\varepsilon_s \colon J_R X$ determines the \emph{optimal move} at position $s$.
\item $q \colon X^\NN \to R$ determines, given a play $\alpha \colon X^\NN$, the outcome of the game.
\item $\omega \colon X^\NN \to\NN$ determines the \emph{relevant part} of a play.
\end{itemize}
The functions $q$ is called the \emph{outcome function}, whereas $\omega$ is called the \emph{control function}. Given a play $\alpha$, all moves $\alpha(i)$ for $i \leq \omega \alpha$ are \emph{relevant moves}. In general, a position $s$ is called relevant if $|s| \leq \omega{\hat s}$, i.e. if in a canonical extension of the current position $s$ the current move is considered a relevant move.
\end{definition}

We shall only consider infinite plays which are obtained by some canonical extension of a finite play $s$. Therefore, we think of these as finite games of unbounded length.  

The intuition behind Definition \ref{seqgames} is as follows. We think of the selection functions $\varepsilon_s$ as specifying at position $s$ what an optimal move at that point would be if we knew the final outcome corresponding to each of the candidate moves. The selection function takes this mapping $X \to R$ of moves to outcomes and tells us what an optimal move would be in that particular case. 

A \emph{strategy} in such game is simply a function $\nexts : X^* \to X$ which determines for each position $s$ what the next move $\nexts(s)$ should be. \emph{To follow a strategy} from position $s$ means to play all following moves according to the strategy, i.e. we obtain a sequence of moves $\alpha(0), \alpha(1), \ldots$ as
\[ \alpha(i) = \nexts(s * \initSeg{\alpha}{i-1}). \]
We call this the \emph{strategic extension of $s$}. The strategic extension of the empty play is called the \emph{strategic play}. 

\begin{definition}[Optimal strategies] A strategy is said to be \emph{optimal}\footnote{This is a stronger notion than the one introduced in \cite{EO(2011A)} for the more general case where the quantifiers are not necessarily attainable.} if the move played at each relevant position $s$ is the one recommended by the selection function $\varepsilon_s$, i.e.
\begin{equation} \label{opt-strategy}
\nexts(s) = \varepsilon_s(\lambda x . q(s * x * \alpha))
\end{equation}
where $\alpha$ is the strategic extension of $s * x$.  
\end{definition}

The main result of \cite{EO(2011A)} is that the product of selection functions computes optimal strategies:

\begin{theorem}[\cite{EO(2011A)}] \label{main-strategy} Given a game $(\varepsilon, q, \omega)$, the strategy
\begin{equation} \label{opt-str-def}
\nexts(s) \stackrel{X}{=} \left( \Prod{s}{\omega}{\varepsilon}(q) \right)_0
\end{equation}
is optimal, and, moreover,  
\begin{equation} \label{strategic-play}
\alpha \stackrel{X^\NN}{=} \Prod{s}{\omega}{\varepsilon}(q_s)
\end{equation}
is the strategic extension of $s$, i.e. $\alpha(n) = \nexts(s * \initSeg{\alpha}{n})$.
\end{theorem}
\begin{proof} We have that
\eqleft{
\begin{array}{lcl}
	\alpha(n) & \stackrel{(\ref{strategic-play})}{=} & \Prod{s}{\omega}{\varepsilon}(q_s)(n) \\[2mm]
	& \stackrel{\textup{L}\ref{main-lemma}}{=} & \Prod{s * \initSeg{\alpha}{n}}{\omega}{\varepsilon}(q_{s * \initSeg{\alpha}{n}})(0) \\[2mm]
	& \stackrel{(\ref{opt-str-def})}{=} & \nexts(s * \initSeg{\alpha}{n}),
\end{array}
}
which proves the second claim. Hence, assuming $s$ is a relevant position, i.e. $(*) \; \omega(\hat s) \geq |s|$ we have
\eqleft{
\begin{array}{lcl}
	\nexts(s) & \stackrel{(\ref{opt-str-def})}{=} & \left( \Prod{s}{\omega}{\varepsilon}(q_s) \right)_0 \\[2mm]
	& \stackrel{(\ref{Prod-var}) + (*)}{=} & \varepsilon_s(\lambda x . \selEmb{\Prod{s * x}{\omega}{\varepsilon}}(q_{s * x})) \\[2mm]
	& = & \varepsilon_s(\lambda x . q(s * x * \beta))
\end{array}
}
where $\beta = \Prod{s * x}{\omega}{\varepsilon}(q_{s * x})$, by the second claim just proven above, is the strategic extension of $s * x$. Hence, we have shown (\ref{opt-strategy}).
\end{proof}

Therefore in this sense the main Theorem \ref{eps-main} characterises $\Prods$ as a procedure that computes an optimal strategy in the game defined by $(\varepsilon,q,\omega)$. We now show that the product also appears naturally in proof theory, with the advantage that it can be related back to the language of sequential games.   

\section{The dialectica interpretation of classical proofs}
\label{sec-main}

We now show how selection functions and their product are intrinsically connected to the functional interpretation of classical proofs. The key observation is that the language of selection functions elegantly captures the way in which the dialectica interpretation treats double negations in negative-translated formulas. In particular the product directly interprets the double negation shift that arises from the negative translation of the axiom of countable choice.

This means that in many cases the algorithms extracted from classical proofs can be easily phrased in the intuitive language of sequential games. Moreover, though couched in the language of higher type recursive functionals, these games often have a natural informal reading in terms of strategic set-theoretic constructions, making the \emph{mathematical} meaning of the extracted program more perspicuous. 

\subsection{Interpreting $\Sigma_2$ theorems}
\label{sigma-section}

Suppose are given a $\Sigma_2$ theorem $A\equiv\exists x^X\forall y^Y A_0(x,y)$ where $A_0$ is decidable.  The negative translation of $A$ is equivalent to $\neg\neg\exists x\forall y A_0(x,y)$, and therefore its functional interpretation is given by 
\eqleft{\nd{\nt{A}}{\varepsilon}{p}= A_0(\varepsilon p,p(\varepsilon p)).} 
In other words, the dialectica interpretation eliminates double negations in front of a $\Sigma_2$ formula with a selection function $\varepsilon\colon J_YX$. If the predicate $A_0(x,y)$ is thought of as prescribing `good' outcomes $y$ for a particular move $x$ as described in Section \ref{sec-intro}, then $\varepsilon$ implements a strategy that selects a move $x=\varepsilon p$ whose outcome with respect to the mapping $p$ is good. 

Thus under the functional interpretation we have the following mapping:
\begin{equation*}\boldsymbol{\begin{array}{rcl}\mbox{ $\Sigma_2$ \bf Theorems} & \mapsto & \mbox{\bf Selection functions}.\end{array}}\end{equation*}
The elimination of double negations in an arbitrary negated formula is essentially a (albeit complex) modular iteration of this process, suggesting to us that selection functions and modes of recursion based on selection functions lie behind the functional interpretation in a fundamental way.

There are several ways of characterising the selection function $\varepsilon$ interpreting $A$. For $\Sigma_2$ theorems Kreisel's \emph{no counterexample interpretation} coincides with the functional interpretation and in this sense the constructive interpretation of $A$ is a selection function $\varepsilon$ that refutes an arbitrary `counterexample' functions $p$. The following example demonstrates how selection functions are fundamental to the functional interpretation of pure classical logic.

\begin{example}[Law of excluded middle]Consider the following simple reformulation of the law of excluded middle for $\Sigma_1$ formulas, better known as the \emph{drinkers paradox}:
\begin{equation} \label{eqn-drinkers}
\DP \quad \colon \quad \exists x^X (\exists y P(y)\to P(x)).
\end{equation}
Note that $\DP$ is intuitionistically equivalent to the $\Sigma_2$ theorem $\exists x\forall y (P(y)\to P(x))$. We ineffectively justify the principle by defining $$x:=\left\{\begin{array}{ll}y & \mbox{for some $y$ satisfying $P(y)$} \\[2mm] 0_X & \mbox{if no such $y$ exists}\end{array}\right. .$$ On the other hand, we can \emph{effectively} justify the principle with the selection function
\begin{equation} \label{eqn-dpeps}
\varepsilon p :=
\left\{
\begin{array}{ll}
p(0_X) & \mbox{if }P(p(0_X)) \\[2mm]
0_X & \mbox{if }\neg P(p(0_X))
\end{array}
\right. 
\end{equation}
that witnesses its functional interpretation:
\begin{equation} \label{eqn-nddrinkers}
\nd{\DP}{\varepsilon}{p}=P(p(\varepsilon p))\to P(\varepsilon p).
\end{equation}
The drinkers paradox is essentially the law of excluded middle applied to the $\Sigma_1$-formula $\exists y P(y)$ i.e. 
\begin{equation*}\label{eqn-dp} \exists b^\BB(b=0\leftrightarrow \exists y P(y))\end{equation*}
where the boolean $b$ is given by $P(x)$. The mapping $p\colon X\to X$ in the functional interpretation of $\DP$ can be seen as a \textit{counterexample function} that attempts to witness $\neg\DP$ i.e.
\begin{equation} \label{eqn-ndrinkers}
\forall x(\neg P(x) \wedge P(px)).
\end{equation}
The constructive version of the law of excluded middle given by its functional interpretation is the statement that for any $n$, $p$ there exists an element $x$ refuting (\ref{eqn-ndrinkers}):
\begin{equation*}\label{eqn-nddp}
P(x)\vee\neg P(px).
\end{equation*}
The selection function $\varepsilon$ witnesses this statement.

\end{example}

One can alternatively view the selection function $\varepsilon$ interpreting $A$ as an algorithm that produces an arbitrary large \emph{approximation} to the ineffective object $x$ satisfying $\forall y A_0(x,y)$. In fact, when $Y = \NN$ the formula $A$ is equivalent to $\exists x^X\forall y \forall i \leq y A_0(x, i)$. Hence, the functional interpretation of $A$ is equivalent to the existence of a selection functions $\varepsilon$ satisfying 
\begin{equation*}\forall p \; \forall i\leq p(\varepsilon p) A_0(\varepsilon p,i).\end{equation*}
We see $p$ as a function that specifies in advance how we want to use $A$ in a particular computation, and $\varepsilon$ returns a sufficiently high quality approximation to $x$. This reading is closer to the notion of a `finitization' of as discussed by Tao in \cite{Tao(2007)}, in the sense that we interpret the qualitative statement that there exists some $x$ with the permanent property $\forall y A_0(x,y)$ by the quantitative statement that there exist approximations $x$ with the temporary property $\forall i\leq px A_0(x,i)$ for arbitrary $p$.

\begin{example}[Convergence and metastability]The functional interpretation of Cauchy convergence
\begin{equation*}\label{eqn-cauchy}\forall k>0\exists n\forall i,j\geq n \; (\|x_i-x_j\|\leq 2^{-k})\end{equation*}
is a sequence of selection functions $(\varepsilon_k)$ that satisfy
\begin{equation}\label{eqn-cauchynd} \forall k>0,p \; \forall i,j\in [\varepsilon_kp,\varepsilon_kp+p(\varepsilon_kp)] \; (\|x_i-x_j\|\leq 2^{-k}).\end{equation}
In other words, the Cauchy convergence property is equivalent to the existence of a sequence of selection functions $\varepsilon_k$ that compute regions of approximate stability, or \emph{metastability}, of size specified by $p$.

This reformulation of convergence plays a key role in ergodic theory, where one obtains quantitative versions of convergence theorems by extracting explicit bounds on $\varepsilon_kp$ that are highly uniform with respect to $(x_n)$. A simple example is the so-called `finite convergence principle' discussed in \cite{Kohlenbach(2008),Tao(2007)}, where one can easily show that given $k$ and $p$ a bounded monotone sequence $$0\leq x_0\leq x_1\leq\ldots\leq 1$$ experiences a period of metastability bounded uniformly by $\tilde p(2^k+1)(0)$ for $\tilde p(n):+p(n)$. 

A more involved example of the extraction of uniform bounds on the selection functions is, for instance, the quantitative mean ergodic theorem proved by Avigad et al. in \cite{Avigad(2010A)}.\end{example} 

\subsection{Interpreting the axiom of choice}
\label{choice-section}

Classical predicate logic $\PLomega$ can be extended to encompass most of mathematics through the addition of choice principles. In particular  the principle of \emph{finite} choice $\FC$ is known to be equivalent to induction and therefore we can define Peano arithmetic (assuming a minimal amount of arithmetic) as $\PAomega:=\PLomega+\FC$, while the further addition of countable choice $\AC$ yields a theory sufficient to formalise a large portion of analysis.

Thus a key part of understanding the computational content of classical proofs is to understand the computational interpretation of the axiom of countable choice combined with classical logic.

Let us first consider an instance of $\AC$ for $\Pi_1$ formulas:
\begin{equation*}\ACPi{1} \ \colon \ \forall n\exists x^X\forall y^Y A_n(x,y)\to\exists f^{\NN\to X}\forall n,y A_n(f n,y),\end{equation*}
for decidable $A_n$. Its negative translation is equivalent to
\begin{equation*}\forall n\neg\neg\exists x\forall y A_n(x,y)\to\neg\neg\exists f\forall n,y A_n(f n,y),\end{equation*}
and its dialectica interpretation is equivalent (using just Markov's principle, which is admitted by the dialectica interpretation) to the statement
\begin{equation}\label{eqn-ndac}\forall\varepsilon,q,\omega\exists f (\forall n,p A_n(\varepsilon_np,p(\varepsilon_np))\to\forall i\leq\omega f A_i(f i,q f)).\end{equation}
This constructive interpretation of $\AC$ asks for a selection function  $F^\varepsilon\colon J_{Y\times\NN}X^\NN$ producing an approximation to the sequence $f$, given selection functions $\varepsilon$ interpreting its premise. Such a selection function can be given by
\eqleft{F^\varepsilon(q,\omega):=\Prod{\pair{}}{\omega}{\varepsilon}(q).}
We now prove in detail that the product of selection functions directly realizes the functional interpretation of the axiom of choice. 

\begin{theorem} \label{thm-choice} The following hold:

\begin{enumerate} 

\item[(a)] The functional interpretation of the schema of \textit{finite choice}
\eqleft{\FC \ \colon \ \forall n\leq N\exists x^X A_n(x)\to\exists s^{X^\ast}\forall n\leq N A_n(s_n)}
is directly witnessed by the finite simple product of selection functions (i.e. $\omega$ a constant function).

\item[(b)] The functional interpretation of the schema of \textit{countable choice}
\eqleft{\AC \ \colon \ \forall n\exists x^X A_n(x)\to\exists f^{X^\NN}\forall n A_n(f n)}
is directly witnessed by the (unbounded) simple product of selection functions.

\item[(c)] The functional interpretation of the schema of \textit{dependent choice}
\eqleft{\DC \ \colon \ \forall s^{X^\ast}\exists x^X A_s(x)\to\exists f^{X^\NN}\forall n A_{\initSeg{f}{n}}(f n)}
is directly witnessed by the dependent product of selection functions.

\end{enumerate}\end{theorem}
\begin{proof} We prove \emph{(c)}, the other parts are particular cases of this. Since $A_s(x)$ is equivalent to $\exists\tilde{x}^{\tilde{X}}\forall y\nd{\nt{A_s(x)}}{\tilde{x}}{y}$ it suffices to interpret $\DC$ for $\Sigma_2$-formulas
\eqleft{\mbox{$\Sigma_2$-}\DC \ \colon \ \forall s^{X^\ast}\exists x^X,\tilde{x}^{\tilde{X}}\forall y\nd{\nt{A_s(x)}}{\tilde{x}}{y}\to\exists f^{X^\NN}\forall n \exists\tilde{x}^{\tilde{X}}\forall y\nd{\nt{A_{\initSeg{f}{n}}(f n)}}{\tilde{x}}{y}.}
Moreover, by adding a dummy variable $t$ of type $\tilde{X}^\ast$ and concatenating the types $X$, $\tilde{X}$ this follows directly from an instance of $\Pi_1$-$\DC$ i.e.
\eqleft{\mbox{$\Pi_1$-}\DC \ \colon \ \forall s^{X^\ast},t^{\tilde{X}^\ast}\exists {x},\tilde{x}\forall y \nd{\nt{A_s(x)}}{\tilde{x}}{y}\to\exists f^{X^\NN},\tilde{f}^{\tilde{X}^\NN}\forall n,y \nd{\nt{A_{\initSeg{f}{n}}(f n)}}{\tilde{f} n}{y}.}
Therefore it suffices to deal with $\Pi_1$-$\DC$, which in general has a negative translation equiavalent to
\eqleft{\nt{\mbox{$\Pi_1$-}\DC} \ \colon \ \forall s^{X^\ast}\neg\neg\exists x^X\forall y A_s(x,y)\to\neg\neg\exists f\forall n,y A_{\initSeg{f}{n}}(f n,y).}
The dialectica interpretation of $\nt{\mbox{$\Pi_1$-}\DC}$ is equivalent to
\begin{equation}
\nd{\nt{\DC}}{F,p,s}{\varepsilon,\omega,q} \;\equiv\; A_s(\varepsilon_sp,p(\varepsilon_sp))\to A_{\initSeg{F}{\omega F}}(F(\omega F),qF),
\end{equation}
omitting, for the sake of readability, the parameters $\varepsilon, \omega$ and $q$ from the functions $F, p$ and $s$. In fact, these parameters $(\varepsilon,\omega,q)$ define a sequential game in the sense of Definition \ref{seqgames}. Therefore, let
\eqleft{
\begin{array}{lcl}
F & {=} & \Prod{\emptyseq}{\omega}{\varepsilon}(q) \\[2mm]
p_s(x) & {=} & \selEmb{\Prod{s * x}{\omega}{\varepsilon}}(q_{s * x}).
\end{array}
}
By Theorem \ref{eps-main} we have that $F$ and $p:=p_{\initSeg{F}{\omega F}}$ and $s:=\initSeg{F}{\omega F}$ are such that $\varepsilon_sp = F(\omega F)$ and $p(\varepsilon_sp) = q F$, and hence, clearly witness $\nd{\nt{\DC}}{F,p,s}{\varepsilon,\omega,q}$.
An analogous but simpler argument proves $(1)$ and $(2)$, proofs of which can also be found in \cite{EOP(2011)} and \cite{EO(2010A)} respectively.\end{proof}

Theorem \ref{thm-choice} proves that under the functional interpretation we have a mapping
\begin{equation*}\boldsymbol{\begin{array}{rcl}\mbox{\bf Choice principles} & \mapsto & \mbox{\bf Product of selection functions}.\end{array}}\end{equation*}
At first glance it may seem strange that an operation that computes optimal strategies in sequential games is related to the axiom of choice is this manner. But if we take a closer look, the game theoretic behaviour of (\ref{eqn-ndac}) becomes clear. The selection functions $\varepsilon_n$ which realise the premise of (\ref{eqn-ndac}) can be seen as a collection of strategies each witnessing the $\Sigma_2$ theorems $(A_n)$. The dialectica interpretation calls for a procedure that takes these pointwise strategies and produces a co-operative selection function $F$ that witnesses $\forall n A_n$. Such a procedure is provided naturally by the product of selection functions.

In the following examples we illustrate how the interpretation of theorems that make direct use of the axiom of choice can be given an intuitive game-theoretic constructive interpretation by the product of selection functions.

\begin{example}[Arithmetic comprehension]We first give a realizer for the functional interpretation of arithmetic comprehension for $\Sigma^0_1$ formulas, which states that for any $\Sigma_1$ predicate $\varphi$ over $\NN$ there exists a set $X$ with
\[ \forall n (n\in X \Leftrightarrow \exists y \pred(n,y)). \]
Computing such $X$ is in general not possible. We can, however, try to compute an ``approximation" to $X$. For instance, we might ask for an $\tilde X$ which only works for a finite number of $n$'s, or an approximation which only checks the existence of $y$'s up to a certain bound (possibly depending on the approximating set $\tilde X$). We call these calibrations of the `size' and `depth' of $X$, respectively.

Arithmetic comprehension follows from the formal statement
\begin{equation*}
\exists f^{\NN\to\NN}\forall n(\exists y \pred(n,y)\to\exists k< fn\pred(n,k)),
\end{equation*}
where we define $X:=\{n \; | \; \exists k<fn\pred(n,k)\}$. Again, we cannot (in general) effectively construct $X$. Indeed, the above is a direct consequence of countable choice applied the non-constructive statement
\begin{equation}\label{eqn-caacprem}\forall n\exists x^\NN(\exists y\pred(n,y)\to\exists k< x\pred(n,k)).\end{equation}
But this is just a collection of instances of $\DP$ applied to the formulas $P_n(x):=\exists k< x\pred(n,k)$. Therefore defining the sequence of selection functions $(\varepsilon_n)$ by
\begin{equation*} \label{eqn-caeps}
\varepsilon_n p:=
\left\{
\begin{array}{ll}
0 & \mbox{if }\forall k<p(0)\negpred(n,k) \\[2mm]
p(0) & \mbox{if }\exists k< p(0)\pred(n,k) 
\end{array}\right. 
\end{equation*}
we have 
\begin{equation*}\exists k< p(\varepsilon_np)\pred(n,k)\to\exists k< \varepsilon_np\pred(n,k)\end{equation*}
for any $n$, $p$, and thus by Theorem \ref{eps-main}, for any counterexample functionals $\omega$, $q$, setting $F:=\Prod{\pair{}}{\omega}{\varepsilon}(q)$ we have
\begin{equation}\label{eqn-ndca}\forall i\leq\omega F(\exists k< qF\pred(i,k)\to\exists k< Fi\pred(i,k))\end{equation}
which is equivalent to the functional interpretation of $\CASig$. So what is the game-theoretic interpretation of our realizer $F$? If we unravel (\ref{eqn-ndca}) we see that we are essentially constructing a finite set 
\[ X_F:=\{i\leq\omega F \; | \; \exists k< Fi\pred(i,k)\} \]
that serves as an approximation to $X$ with the property that if $i\leq\omega F$ has a witness for $\varphi$ bounded by $qF$ then we must have $i\in X_F$. In this sense $\omega$ and $q$ can be read as set functions that calibrate the `size' and `depth' respectively of an approximation to $X$.

The set $X_F$ is constructed as an optimal play in the game $(\varepsilon,q,\omega)$. The job of the selection functions at round $n$ is to decide whether or not to include the number $n$ in the approximation, given that it has already made this decision for $\{0,\ldots,n-1\}$. Its default is to omit $n$ by playing $0$, but if the resulting outcome $p_n(0)$ bounds some witness to $n$, it instead adds $n$ and steals this witness as justification.

Therefore in this scenario the product of selection functions forms an intuitive set-theoretic construction, starting with the empty set and strategically adding elements until it reaches the desired approximation. When interpreting a theorem that makes use of arithmetic comprehension as a lemma, we can simply plug in our realizer and impart its game theoretic meaning to better understand the realizer of the main theorem.

Some simple examples of well-known existence theorems that can be given a direct constructive interpretation using this instance of the product can be found in e.g. Simpson \cite{Simpson(99)}, such as the existence of maximal ideas in countable commutative rings or torsion subgroups in countable abelian groups. A more involved consequence of arithmetic comprehension using a more complex game, the Bolzano-Weierstrass theorem, will be discussed in the next section.\end{example}

\begin{example}[No injection $(\NN\to\NN)\to\NN$]\label{ex-noinj}Following \cite{Oliva(2006)} we show that a higher type instance of the product that produces a sequence of \emph{functions} can be used to effectively prove that there is no injection $\Psi\colon (\NN\to\NN)\to\NN$ in any model of functionals in which the unbounded product exists. This time we consider the drinkers paradox applied to the formulas $P_n(f^{\NN\to\NN}):=(n=\Psi f)$. Defining
\begin{equation*} \label{eqn-nieps}
\varepsilon_n p :=
\left\{
\begin{array}{ll}
p(\cZero) & \mbox{if }n=\Psi(p(\cZero)) \\[2mm]
\cZero^{\NN^\NN} & \mbox{if }n\neq \Psi(p(\cZero)),
\end{array}
\right.
\end{equation*}
where $\cZero^{\NN^\NN} = \lambda k . 0$, we have
\begin{equation*}\forall n,p(n=\Psi(p(\varepsilon_nP))\to n=\Psi(\varepsilon_np)).\end{equation*}
As before, by Theorem \ref{eps-main}, setting $F:=\Prod{\pair{}}{\omega}{\varepsilon}(q)$ we obtain
\begin{equation} \label{eqn-ndni}
\forall n\leq\omega F(n = \Psi(qF) \to n = \Psi(F_n)).
\end{equation}
Setting $qF\stackrel{\NN\to\NN}{:=}\lambda k.(F_k(k)+1)$ we have a diagonal function that differs from each $F_k$ at point $k$. Furthermore, if we set $\omega F=\Psi(qF)$, then from (\ref{eqn-ndni}) on $i=\omega F=\Psi(qF)$ we obtain $$\Psi(qF)=\Psi(F_{\Psi(qF)}).$$ But $qF$ and $F_{\Psi(qF)}$ differ by definition, and we're done.

This is simply a constructive version of the usual diagonalisation argument used to prove that there is no injection from the real numbers to the natural numbers. The outcome functional $q\alpha$ is defined so that it differs from each of the $\alpha i$ on at least one value. As before, the $\varepsilon_n$ plays a default value $\cZero$ and looks at the diagonal function $p_n(\cZero)$ obtained by applying $q$ to the optimal continuation of this move. If $\Psi(p_n(\cZero))$ it steals this witness and sets $\alpha n=p_n(\cZero)$, else it is not concerned and sticks with the default move. The idea is to construct a sequence of representative functions $\alpha$ such that if $\Psi h$ for some $h$, then we must also have $\Psi(\alpha n)$. Of course we cannot effectively produce such an $\alpha$, but using the product of selection functions we \textit{can} produce an approximation that works at the point $\Psi(q\alpha)$, which is actually all that we need.
\end{example}

\subsection{The product versus standard modes of recursion}

A consequence of Theorem \ref{thm-choice}, and the fact that classical arithmetic and analysis can be formulated as classical logic plus finite and countable choice respectively, is that the functional interpretation of classical proofs can be given entirely in terms of the product of selection functions. In fact we can reformulate Theorem \ref{thm-soundness} as follows, although we omit the details here and encourage the reader to consult \cite{EO(2010A),EOP(2011)} instead.

\begin{theorem} \label{thm-soundnesseps} We have

\begin{itemize}

\item[(a)]$\PAomega$ has a functional interpretation in primitive recursive arithmetic plus the finite product of selection functions (see \cite{EOP(2011)} for details).

\item[(b)]$\PAomega+\AC$ has a functional interpretation in primitive recursive arithmetic plus the unbounded product of selection functions.

\end{itemize}\end{theorem}

It is natural then to ask how the product compares to those modes of recursion typically used in the functional interpretation of arithmetic and analysis.

G\"{o}del's primitive recursive functionals of finite type \cite{Goedel(58)} are the functional analogue of induction. By Theorem \ref{thm-choice} (a) the finite product of selection functions is a functional analogue of finite choice, which is known to be equivalent to induction \cite{Parsons(70)}. In \cite{EOP(2011)} it is shown that the finite product is in fact equivalent to G\"{o}del's primitive recursors over a weak $\lambda$-calculus, and thus offers an alternative construction of system $\systemT$.

Countable choice and dependent choice are typically interpreted using Spector's bar recursion \cite{Spector(62)}. By Theorem \ref{thm-choice} (b) and (c) we see that these are also interpreted by the unbounded product, and in \cite{EO(2010A)} it is shown that bar recursion is primitive recursively equivalent to the unbounded product. The whole picture is sketched in Figure \ref{fig-int}.

\begin{figure}[h]
\begin{center}{\small
\[\xymatrix{{{\mbox{Induction}}}\ar[rr]^{\textup{\cite{Goedel(58)}}}\ar@{<=>}[d]_{\textup{\cite{Parsons(70)}}} & & {{\mbox{Primitive recursion}}}\ar@{<=>}[d]^{\textup{\cite{EOP(2011)}}} \\ {{\mbox{Finite choice}}}\ar[rr]_{\textup{\cite{EOP(2011)}}} & & {{\mbox{Finite product}}} \\ & & {{\mbox{Unbounded product}}}\ar@{<=>}[dd]^{\textup{\cite{EO(2009)}}} \\ {{\mbox{Countable/dependent choice}}}\ar[rru]^{\textup{\cite{Spector(62)}}}\ar[rrd]_{\textup{\cite{EO(2010A)}}} & &  \\ & & {{\mbox{Spector's bar recursion}}}}\]}
\end{center}
\caption{Functional intepretation of arithmetic and analysis}
\label{fig-int}
\end{figure}
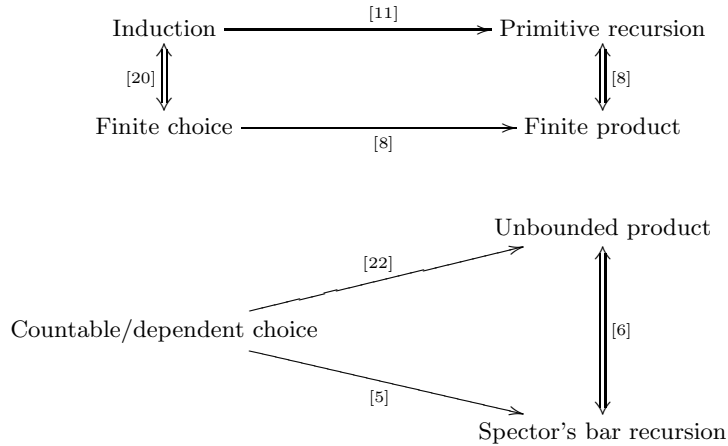 

Of course our point here is that the advantage of using the product as opposed to the other modes of recursion is that it has a highly intuitive semantics, and witnesses extracted using the product often have a clear game-theoretic meaning. This is in stark constrast to other methods, particularly Spector's bar recursion, which are often very difficult to comprehend on a semantic level.

\section{Interpreting the Bolzano-Weierstrass theorem}
\label{sec-bolzano}

In this section we present a case study in which we formally extract a realizer for the functional interpretation of the Bolzano-Weierstrass theorem using the product of selection functions.

The constructive content of this theorem has been studied before, and in particular a detailed analysis using the dialectica interpretation and Spector's bar recursion is given in \cite{Safarik(2010)}. 

Our aim here is to show that, even though the Bolzano-Weierstrass theorem is relatively complex from a logical point of view, one can extract from its proof a program whose behaviour can be clearly understood, at least on an informal level, in terms of optimal strategies in sequential games. 

As in \cite{Safarik(2010)} we analyse a proof of the theorem that combines countable choice with \emph{weak K\"{o}nig's lemma} - the statement that all infinite binary trees $T$ have an infinite branch:
\begin{equation*}\label{eqn-wkl}\WKL \; \colon \; \forall n\exists s^{\BB^{n}} \; T(s)\to\exists\alpha\forall n \; T(\initSeg{\alpha}{n}).\end{equation*}
We use the product of selection functions only to interpret the instance of countable choice used in the proof, as this forms the core of the extracted algorithm. In \cite{Howard(81A)} Howard showed that $\WKL$ has a functional interpretation using only a weak, binary form of bar recursion. Rather than interpreting $\WKL$ using the product of selection functions (as done in \cite{OP(2011B)}), for simplicity we just make use of Howard's realizer as its contribution to the complexity of the overall program is small compared to the main instance of the product.

\subsection{The Bolzano-Weierstrass theorem.}

The Bolzano-Weierstrass theorem states that any bounded sequence in $\RR^n$ has a convergent subsequence. Here we restrict ourselves to sequences of rationals $(x_i)$ in the unit interval $[0,1]$, as our analysis can be readily generalised. In the language of formal analysis, the Bolzano-Weierstrass theorem is given by
\begin{equation*} \label{eqn-bw0}
\BW(x_i) \quad \colon \quad \exists a^{\BB^\NN},b^{\NN^\NN}\forall n(bn<b(n+1)\wedge x_{bn}\in I_{\initSeg{a}{n}}),
\end{equation*}
where for a finite sequence of booleans $s$ we define the interval $$I_s:=\left[\sum_{i=0}^{|s|-1}\frac{s_i}{2^{i+1}} \; , \; \sum_{i=0}^{|s|-1}\frac{s_i}{2^{i+1}}+\frac{1}{2^{|s|}}\right]$$ for $|s|>0$ and $I_{\pair{}}:=[0,1]$. Intuitively $a$ encodes a limit point $\bar{a}:=\sum^{\infty}\frac{a_i}{2^{i+1}}$ of the sequence $(x_i)$ and $b$ defines a subsequence converging to this limit point, where $|x_{bn}-\bar{a}|\leq 2^{-(n-1)}$ for all $n$.

The functional interpretation of the Bolzano-Weierstrass theorem is given by
\begin{equation*} \label{eqn-bwnd}
\nd{\BW(x_i)}{A,B}{\psi} = \forall n \leq \psi AB (Bn < B(n+1)\wedge x_{B n}\in I_{\initSeg{A}{n}}),
\end{equation*}
where, to easy readability, we are omitting the dependency of $A$ and $B$ on $\psi$. The interpreted theorem states that there exist arbitrary large finite approximations $B$ to a convergent subsequence, in the sense that $|x_{Bn}-\bar{A}|\leq 2^{-(n-1)}$ for all $n\leq\psi AB$.

\subsection{A formal proof of $\BW(x_i)$.}

Assume an infinite sequence of rationals $(x_i)_{i \in \NN}$ is fixed. Let us prove theorem $\BW(x_i)$, i.e.
\[ \exists a^{\BB^\NN},b^{\NN^\NN}\forall n(bn<b(n+1)\wedge x_{bn}\in I_{\initSeg{a}{n}}) \]
directly using $\WKL$ and $\AC$. We define the predicate $T$ by
\begin{equation} \label{tree-def} 
T(s^{\BB^\ast},k) \; := \; |s|<k\wedge\exists i\in (|s|,k] \; \big( x_i\in I_s \big).
\end{equation}
We write $\BB^*$ for finite sequences of booleans, and $\BB^n$ for sequences of booleans of length $n$.

\begin{lemma} \label{lem-monotoneskolem} By countable choice $\AC$ there exists a function $\beta\colon\NN\to\NN$ such that
\begin{equation}
\label{eqn-monotoneskolem}
\forall n,s^{\BB^n}(\exists kT(s,k) \to T(s,\beta n)).
\end{equation}
\end{lemma}
\begin{proof} By the drinkers paradox we have
\[ \forall n,s^{\BB^n}\exists l(\exists k T(s,k)\to T(s,l)). \]
By bounded collection and the fact that $T$ has the monotonicity property $T(s,k)\to T(s,k+l)$ we have 
\begin{equation} \label{eqn-acprem}
\forall n\exists L\forall s^{\BB^n}(\exists k T(s,k)\to T(s,L)).
\end{equation}
Finally, by countable choice we obtain $\beta$ satisfying (\ref{eqn-monotoneskolem}).
\end{proof}

For the rest of the section let $\beta$ be a function whose existence is shown in the above lemma, so $\beta$ satifies (\ref{eqn-monotoneskolem}). Also, define $T^\beta(s):=T(s,\beta(|s|))$, so that $\exists kT(s,k)\leftrightarrow T^\beta(s)$.

\begin{corollary} $T^\beta(s)$ is a decidable tree predicate.
\end{corollary}
\begin{proof} $T^\beta(s)$ is clearly decidable in the given oracle $\beta$. It remains to see that it is prefix-closed. Observe that $T(s\ast t,\beta(|s\ast t|))\to T(s,\beta(|s\ast t|))$, by the definition of $T$. Also, by (\ref{eqn-monotoneskolem}), $T(s,\beta(|s\ast t|))\to T(s,\beta(|s|))$. Combining the two we have $T^\beta(s * t) \to T^\beta(s)$.
\end{proof}

Hence, given a decidable binary tree predicate, we can apply weak K\"onig's lemma to obtain the following:

\begin{lemma} \label{lem-bwwkl}There exists a sequence $a \colon \BB^\NN$ such that 
\begin{equation} \label{eqn-bwwkl}
\forall n \underbrace{\left(n<\beta n\wedge\exists i\in (n,\beta n]  \big( x_i\in I_{\initSeg{a}{n}} \big) \right)}_{T^\beta(\initSeg{a}{n})}.
\end{equation}
\end{lemma}

\begin{proof} This follows from $\WKL$ applied to $T^\beta$, once we have shown that the tree $T^\beta(s)$ has branches of arbitrary length. To see that, fix $n$ and let $s$, with $|s| = n$, be the index of the interval $I_s$ which contains $x_n$. Then, clearly we have 
\[ |s| < n+1 \wedge x_n \in I_s \]
which implies $T(s, n+1)$. By (\ref{eqn-monotoneskolem}) we obtain $T(s, \beta(|s|))$. For future reference, we call $h$ this function producing $s$ for each given $n$.
\end{proof}

\begin{theorem}[Bolzano-Weierstrass]\label{thm-bw}Given a sequence of rationals $(x_i)_{i \in \NN}$, there exists $a \colon \BB^\NN$ and $b \colon \NN^\NN$ such that 
\begin{equation}\label{eqn-bw}\forall n(bn<b(n+1)\wedge x_{bn}\in I_{\initSeg{a}{n}}).\end{equation}\end{theorem}

\begin{proof} Let $a$ be as in Lemma \ref{lem-bwwkl}. Define $b$ by
\begin{align} \label{eqn-b}
b0 &:= 0 \\ \notag
b(n+1) &:= \mu i\in (bn+1,\beta(bn+1)] \left( x_i\in I_{\initSeg{a}{bn+1}} \right).
\end{align}
By (\ref{eqn-bwwkl}) such least $i \in (bn+1,\beta(bn+1)]$ satisfying $x_i\in I_{\initSeg{a}{bn+1}}$ always exists. Clearly we have $bn<b(n+1)$, and also $x_{b(n+1)}\in I_{\initSeg{a}{bn+1}}\subseteq I_{\initSeg{a}{n+1}}$, since $bn+1\geq n+1$.
\end{proof}

\subsection{Lemma \ref{lem-monotoneskolem} via the product of selection functions}
\label{subsec-bwac}

We first interpret our main instance of countable choice (Lemma \ref{lem-monotoneskolem}). We want to produce an approximation to the function $\beta$ given counterexample functions $\omega,q\colon\NN^\NN\to\NN$ for $n$ and $k$ respectively in (\ref{eqn-monotoneskolem}):
\begin{equation} \label{eqn-ndmonotoneskolem}
\forall\omega,q \exists\beta\forall n\leq\omega\beta\forall s^{\BB^n}(T(s,q\beta)\to T(s,\beta n)).
\end{equation}
First we need to find selection functions $\delta_n \colon \NN^\NN \to \NN$ witnessing the functional interpretation of (\ref{eqn-acprem}):
\begin{equation}\label{eqn-ndacprem}
\exists\delta\forall n,p\forall s^{\BB^n}(T(s,p(\delta_np))\to T(s,\delta_np)).
\end{equation}
Since (\ref{eqn-acprem}) is just the drinkers paradox combined with bounded collection, appropriate selection functions are constructed in a similar manner to (\ref{eqn-dpeps}).

\begin{lemma}\label{lem-acprem}Let $\delta_n$ be defined as
\begin{equation} \label{eqn-delta}
\delta_n p := p^i(0)
\end{equation}
where $i$ is the least $\leq 2^n$ such that, for all $s^{\BB^n}$, $\mt(s, p^{i+1}(0)) \to \mt(s, p^i(0))$. Then $\delta$ witnesses (\ref{eqn-ndacprem}).
\end{lemma}
\begin{proof} Note that (\ref{eqn-ndacprem}) holds by definition once we show that such $i \leq 2^n$ must exist. Assume, for the sake of a contradiction, that
\begin{itemize}
	\item[(I)] for all $i \leq 2^n$ there exists an $s^{\BB^n}$ such that $\mt(s, p^{i+1}(0))$ and $\neg \mt(s, p^i(0))$.
\end{itemize}
Because $\mt(s, k)$ is monotone on the second argument, (I) implies that
\begin{itemize}
	\item[(II)] $p^i(0) < p(p^i(0))$, for all $0 \leq i \leq 2^n$.
\end{itemize}
Since, in (I), we have $2^n +1$ possible values for $i$ but only $2^n$ possible values for $s$, there must be an $s$ and distinct $i$ and $j$, say $i < i+1 \leq j$, such that $\mt(s, p^{i+1}(0))$ and $\neg \mt(s, p^j(0))$. By (II), however, that is a contradiction.
\end{proof}

\begin{theorem} \label{thm-ndmonotoneskolem} Given $q \colon \NN^\NN \to \NN$ and $\omega \colon \NN^\NN \to \NN$ let $\beta \colon \NN^\NN$ be defined as
\[ \beta:=\Prod{\pair{}}{\omega}{\delta}(q), \]
with $\delta$ as in (\ref{eqn-delta}). Then $\beta$ witnesses (\ref{eqn-ndmonotoneskolem}).
\end{theorem}
\begin{proof} By Lemma \ref{lem-acprem}, the $\delta_n$ as defined in (\ref{eqn-delta}) are such that
\[ \forall n,p\forall s^{\BB^n}(T(s,p(\delta_np))\to T(s,\delta_np)). \]
For $n \leq \omega \beta$ and $p = p_{\initSeg{\beta}{n}}$, by Theorem \ref{eps-main} we have $\beta n = \delta_n p$ and $q \beta = p(\delta_n p)$, from which we can conclude $\forall s^{\BB^n}(T(s,q\beta)\to T(s,\beta n))$.
\end{proof}

\subsection{A realizer for $\BW(x_i)$}
\label{subsec-bwnd}

Finally, we show how the instance of the product of selection functions in Theorem \ref{thm-ndmonotoneskolem}, used to interpret the crucial Lemma \ref{lem-monotoneskolem}, lies behind an algorithm for constructing approximations to $\BW$. We first interpret Lemma \ref{lem-bwwkl}, making use of Howard's realizer for $\WKL$ using a weak, binary form of bar recursion, full details of which can be found in \cite{Howard(81A)}.

\begin{lemma}\label{lem-ndbwwkl} For any counterexample function $\varphi\colon\BB^\NN\times \NN^\NN\to\NN$ there exist $\beta \colon \NN^\NN$ and $A \colon \BB^\NN$ satisfying
\begin{equation} \label{eqn-ndbwwkl}
\forall n \leq \varphi A \beta \underbrace{\left( n < \beta n \wedge \exists i \in (n,\beta n] \left( x_i \in I_{\initSeg{A}{n}} \right)\right)}_{T(\initSeg{A}{n},\beta n)}.
\end{equation}
\end{lemma}
\begin{proof} Assume $\varphi \colon\BB^\NN\times \NN^\NN\to\NN$ given. For any given $\beta \colon \NN^\NN$ let $N^\beta \colon \BB^\ast\to\NN$ be defined via Howard's binary bar recursion as
\begin{equation*} \label{eqn-howbar}
N^\beta(t) :=
\left\{
\begin{array}{ll}
0 & \mbox{if $\exists s\preceq t  \left( \varphi(\hat s, \beta) <|s| \right)$} \\[2mm]
1+\max\{N^\beta(t\ast 0),N^\beta(t\ast 1)\} & \mbox{otherwise}.
\end{array}
\right. 
\end{equation*}
We first construct $\beta$ as in Theorem \ref{thm-ndmonotoneskolem} where we set
\begin{equation*}
q \beta = \omega \beta := N^\beta(\pair{})+1,
\end{equation*}
to obtain
\begin{equation} \label{eqn-ndmonotoneskolem0}
\forall n \leq N^\beta(\pair{}) \; \forall s^{\BB^n}(T(s,N^\beta(\pair{})+1)\to T(s,\beta n)).
\end{equation}
Let $N = N^\beta$ for $\beta$ as just defined. It can be shown (cf. Howard \cite{Howard(81B)}) that for any $|t|\geq N(\pair{})$ there is some $s\preceq t$ with $\varphi(\hat s, \beta) <|s|$. Therefore we define 
\begin{equation*}
A := \hat t, \quad \quad \mbox{ where }t=\mu s\preceq h(N(\pair{})) \left( \varphi(\hat s, \beta) <|s| \right)
\end{equation*}
where $h$ is defined as in the proof of Lemma \ref{lem-bwwkl}. Now, by the definition of $h$ we have $T(h(N(\pair{})), N(\pair{})+1)$. Also, for $n \leq \varphi(A, \beta)$, by the definition of $t$ (in definition of $A$) we have $n < |t|$, and hence $\initSeg{A}{n} \preceq h(N(\pair{}))$. Therefore, by the definition of $T(s, k)$ we also have $T(\initSeg{A}{n},N(\pair{})+1)$. Finally, by (\ref{eqn-ndmonotoneskolem0}) we get $T(\initSeg{A}{n},\beta n)$, and so we have proved (\ref{eqn-ndbwwkl}).
\end{proof}

We are in a position now to effectively witness an approximation to the real Bolzano-Weierstrass (Theorem \ref{thm-bw}). 

\begin{theorem} \label{thm-ndbw} For any counterexample function $\psi\colon\BB^\NN\times\NN^\NN\to\NN$ there exists $A \colon \BB^\NN$ and $B \colon \NN^\NN$ satisfying
\[ \underbrace{\forall n \leq \psi AB (Bn < B(n+1)\wedge x_{B n}\in I_{\initSeg{A}{n}})}_{\nd{\BW(x_i)}{A,B}{\psi}}. \]
\end{theorem}
\begin{proof} Let $b_{a, \beta}$ denote the construction in (\ref{eqn-b}), and let 
\[ \tilde\beta(n):=\max_{i\leq n+1}\beta(i). \]
Define
\[ \varphi(A, \beta) := \tilde\beta^{\psi(A, b_{A,\beta})}(0). \]
Construct $A$ and $\beta$ as in Lemma \ref{lem-ndbwwkl} using counterexample function $\varphi$ as above defined, and let $B := b_{A,\beta}$. By (\ref{eqn-ndbwwkl}) we have
\begin{equation} \label{final-approx}
\forall n \leq \tilde\beta^{\psi A B}(0) \left( n < \beta n \wedge \exists i \in (n,\beta n] \left( x_i \in I_{\initSeg{A}{n}} \right)\right).
\end{equation}
This is an approximation of Lemma \ref{lem-bwwkl}, but it is enough to obtain our desired approximation of Theorem \ref{thm-bw}. Indeed, for all $n \leq \psi A B$ we are guaranteed to have, by (\ref{final-approx}), that $B n < B(n+1)$ and $x_{B n} \in I_{\initSeg{A}{n}}$.
\end{proof}

\subsection{Understanding the realizer for $\BW(x_i)$}
\label{subsec-bwunderstand}

Let us first take an informal look at the game that forms the basis of the interpretation of countable choice given in Section \ref{subsec-bwac}. The strategy at round $n$ is to pick a number $m$ satisfying
\begin{equation*}
\exists j\in (n,p(m)] \left( x_j\in I_s \right) \to \exists j\in (n,m] \left( x_j\in I_s \right)
\end{equation*}
for all $s \colon \BB^n$. In other words, given a local outcome function $p \colon \NN \to \NN$, the selection function picks a number $m = \delta_n p$ such that whenever the interval $I_s$ contains some $x_j$ where $j$ in bounded by the outcome $p(m)$ of playing $m$, then $I_s$ also contains some $x_j$ with $j$ bounded by $m$.

The selection function $\delta_n$ prescribes what is essentially an iterated version of the strategy given for the drinkers paradox (\ref{eqn-dpeps}). It first attempts to play $m=0$ provided that $\forall s^{\BB^n} \forall j\in (n,p(0)] \left( x_j\notin I_s \right)$. But this cannot be the case unless $p(0) \leq n$. Hence, if $p(0) > n$, there exists an $s$ and a $j \in (n, p(0)]$ with $x_j \in I_s$. It then continues and attempts to play $p(0)$ in the hope that no additional intervals $I_s$ contain $x_j$ for $j\in (p(0),p^2(0)]$. Continuing along these lines, it is not difficult to see that $m=p_n^k(0)$ must work for some $k\leq 2^n$ since there are only $2^n$ intervals $I_s$, with $|s|$.

The resulting optimal strategy in the game $(\delta,\omega,q)$ is a sequence $\beta$ that acts as an approximation to the function ineffectively constructed in Lemma \ref{lem-monotoneskolem}:
\begin{equation*}
\forall n \leq \omega \beta \; \forall s^{\BB^n}(\exists k\in (n,q\beta] \left( x_j \in I_s \right) \to \exists j \in (n,\beta n] \left( x_j\in I_s \right) ).
\end{equation*}
The power of this procedure is evident when we observe that by constructing the outcome function $q$ and control function $\omega$ as in Section \ref{subsec-bwnd} (incorporating Howard's realizer for $\WKL$), the resulting optimal play $\beta$ can be directly used to construct approximations $A$ and $B$ for $\BW(x_i)$.

Our aim here has been to convince the reader that while constructing a realizer for the functional interpretation of the Bolzano-Weierstrass theorem takes a reasonable amount of work, the game theoretic intuition behind the product of selection functions allows us to gain a better understanding of the key operational features of this realizer.

\section{Further remarks}
\label{sec-further}

In this article we have shown that the language of selection functions and sequential games underlies the dialectica interpretation of classical proofs in a fundamental way, and we have used the product of selection functions to construct a concise and intuitive computational interpretation of some well-known theorems.

Our motivation has been a more qualitative understanding of functional interpretations, as a response the fact that formal proof-theoretic methods are becoming increasingly relevant in modern mathematics. We have shown that the product of selection functions is a fundamental construction behind the dialectica interpretation of classical proofs, and we hope to have convinced the reader that in practise it leads to extracted programs that have an expressive reading in terms of optimal strategies in sequential games.

There is a lot of work to be done towards understanding formal proof-theoretic techniques in mathematical terms, and the question of adapting and refining functional interpretations so that they can be seen as intelligent translations on mathematical proofs as opposed to just syntactic translations on logic sentences forms a very interesting area of research. The authors believe that there are several potentially fruitful avenues for further research.

One is to explore in more detail the link between the dialectica interpretation and the closely related `correspondence principle' implicitly used in areas like ergodic theory. In particular, the finitary version of theorems discussed by Tao in \cite{Tao(2007),Tao(2008)} are strictly speaking related to the \emph{monotone} variant of the dialectica interpretation, which extracts uniform bounds, or majorants, for realizers of the interpretation. It would be interesting to try to gain an understanding of how the product of selection functions can be used to extract realizers for the monotone interpretation and therefore produce constructive proofs of theorems that can truly be seen as `finitizations' in the sense of Tao.

Another interesting issue is the \emph{efficiency} of the product of selection functions in producing a realizer. For instance, a quick analysis of Example \ref{ex-noinj} shows that if $\Psi(\lambda n.0)=\Psi(\lambda n.1)$, our program potentially misses this obvious counterexample and eventually produces a much more elaborate one. This highlights the fact that while the product is indeed an intuitive realizer for the axiom of choice, it is far from optimal and refinements of the procedure or even completely different recursion schemata may be more suited to interpreting specific principles. An example of this is open recursion proposed by Berger in \cite{Berger(2004)} for the realizability interpretation of the minimal bad sequence argument. 

A related question is the efficiency of the dialectica interpretation itself, and the comparison of extracted programs to those obtained using other proof interpretations such as modified realizability. In particular, the modified realizability interpretation of choice has an interesting realizer given by Berardi et al. \cite{Berardi(98)} that was also shown to have a natural game theoretic reading.

We conclude with the remark that our paper belongs to a larger body of recent work on the theory of selection functions and sequential games by the first author and M. Escard{\'o}, starting with \cite{EO(2009)} and surveyed in \cite{EO(2011A)}. Particularly relevant is \cite{OP(2011B)}, in which the product of selection functions is used to extract a game theoretic realizer for Ramsey's theorem. This can be seen as an extended case study illustrating the methods employed here.

\bibliographystyle{plain}

\bibliography{dblogic}

\end{document}